\documentclass[11pt]{amsart}
\usepackage[margin=37.5mm]{geometry}
\usepackage{amsmath,amssymb}
\usepackage{amsthm}

\newtheorem{thm}{Theorem}[section]

\newtheorem{lem}{Lemma}[section]
\newtheorem{cor}{Corollary}[section]
\newtheorem{defi}{Definition}[section]
\newtheorem{ex}{Example}[section]
\newtheorem{rem}{Remark}[section]

\begin{document}

\title{Generic and dense distributional chaos with shadowing}
\author{Noriaki Kawaguchi$^\ast$}
\thanks{$^\ast$JSPS Research Fellow}
\address{Faculty of Science and Technology, Keio University, 3-14-1 Hiyoshi, Kohoku-ku, Yokohama, Kanagawa 223-8522, Japan}
\email{gknoriaki@gmail.com}

\begin{abstract}
For continuous self-maps of compact metric spaces, we consider the notions of generic and dense chaos introduced by Lasota and Snoha and their variations for the  distributional chaos, under the assumption of shadowing. We give some equivalent properties to generic uniform (distributional) chaos in terms of chains and prove their equivalence to generic (distributional) chaos under the shadowing. Several examples illustrating the main results are also given.
\end{abstract}

\maketitle
\markboth{NORIAKI KAWAGUCHI}{GENERIC AND DENSE DISTRIBUTIONAL CHAOS WITH SHADOWING}

\section{Introduction}

Shadowing is a feature of topologically hyperbolic dynamical systems, and its implications for chaos is a subject of ongoing research \cite{AC,K2,K3,LLT,OW} (see \cite{AH,P} for a general background). One of the definitions of chaos is the generic chaos proposed by Lasota (see \cite{Pi}). It means the genericity of Li-Yorke pairs and so concerns the notion of Li-Yorke chaos. Inspired by Lasota's definition, in \cite{S1}, Snoha introduced its variations: the generic $\delta$-chaos, dense $\delta$-chaos, and dense chaos, here $\delta>0$. Another definition derived from Li-Yorke chaos is the distributional chaos by Schweizer and Sm\'ital \cite{SS}, and it has natural analogs of generic and dense chaos. In this paper, we consider these notions of chaos for general topological dynamical systems with the shadowing property.

We begin with some definitions concerning the shadowing property. Throughout, $X$ denotes a compact metric space endowed with a metric $d$. 

\begin{defi}
\normalfont
Given a continuous map $f\colon X\to X$, a finite sequence $(x_i)_{i=0}^{k}$ of points in $X$, where $k$ is a positive integer, is called a {\em $\delta$-chain} of $f$ if $d(f(x_i),x_{i+1})\le\delta$ for every $0\le i\le k-1$. Let $\xi=(x_i)_{i\ge0}$ be a sequence of points in $X$. For $\delta>0$, $\xi$ is called a {\em $\delta$-pseudo orbit} of $f$ if $d(f(x_i),x_{i+1})\le\delta$ for all $i\ge0$. For $\epsilon>0$, $\xi$ is said to be {\em $\epsilon$-shadowed} by $x\in X$ if $d(f^i(x),x_i)\leq \epsilon$ for all $i\ge 0$. We say that $f$ has the {\em shadowing property} if for any $\epsilon>0$, there is $\delta>0$ such that every $\delta$-pseudo orbit $\xi$ of $f$ is $\epsilon$-shadowed by some $x\in X$.
\end{defi}

Then, we give the definitions of generic and dense (distributional) chaos (see, for example, \cite{OZ} for a background).

\begin{defi}
\normalfont
Let $f\colon X\to X$ be a continuous map. A pair $(x,y)\in X^2$ is said to be {\em $\delta$-scrambled} for $\delta>0$ if

\[
\limsup_{i\to\infty}d(f^i(x),f^i(y))\ge\delta\quad\text{and}\quad\liminf_{i\to\infty}d(f^i(x),f^i(y))=0. \tag{C}
\]

For $(x,y)\in X^2$ and $t>0$, we define
\[
F_{xy}(t)=\liminf_{n\to\infty}\frac{1}{n}|\{0\le i\le n-1\colon d(f^i(x),f^i(y))<t\}|
\]
and
\[
F^\ast_{xy}(t)=\limsup_{n\to\infty}\frac{1}{n}|\{0\le i\le n-1\colon d(f^i(x),f^i(y))<t\}|.
\]
A pair $(x,y)\in X^2$ is said to be  {\em DC1-$\delta$-scrambled} (resp. {\em DC2-$\delta$-scrambled}) for $\delta>0$ if
\[
F_{xy}(\delta)=0\quad\text{and}\quad F^\ast_{xy}(t)=1\quad\text{for all}\quad t>0, \tag{DC1}
\] 
\[
F_{xy}(\delta)<1\quad\text{and}\quad F^\ast_{xy}(t)=1\quad\text{for all}\quad t>0. \tag{DC2}
\]

For $\delta>0$, we denote by ${\rm C}_\delta(X,f)$ (resp. ${\rm DC1}_\delta(X,f)$, ${\rm DC2}_\delta(X,f)$) the sets of $\delta$-scrambled (resp. DC1-$\delta$-scrambled, DC2-$\delta$-scrambled) pairs for $f$. By definitions, we have
\[
{\rm DC1}_\delta(X,f)\subset {\rm DC2}_\delta(X,f)\subset {\rm C}_\delta(X,f).
\]

For $\Sigma\in\{{\rm C},{\rm DC1},{\rm DC2}\}$, let $\Sigma(X,f)=\bigcup_{\delta>0}\Sigma_\delta(X,f)$. Then, we say that $f$ exhibits
\begin{equation*}
\begin{aligned}
gu\Sigma&\quad\text{if}\quad\text{$\Sigma_\delta(X,f)$ is a residual subset of $X^2$ for some $\delta>0$}, \\
du\Sigma&\quad\text{if}\quad\text{$\Sigma_\delta(X,f)$ is a dense subset of $X^2$ for some $\delta>0$}, \\
g\Sigma&\quad\text{if}\quad\text{$\Sigma(X,f)$ is a residual subset of $X^2$}, \\
d\Sigma&\quad\text{if}\quad\text{$\Sigma(X,f)$ is a dense subset of $X^2$}.
\end{aligned}
\end{equation*}
\end{defi}

\begin{rem}
\normalfont
For every $\sigma\in\{gu,du,g,d\}$, 
\[
\sigma{\rm DC1}\Rightarrow\sigma{\rm DC2}\Rightarrow\sigma{\rm C}.
\]
Also, for each $\Sigma\in\{{\rm C},{\rm DC1},{\rm DC2}\}$, $gu\Sigma$ implies $du\Sigma$ and $g\Sigma$; and $g\Sigma$ implies $d\Sigma$.
\end{rem}

Here, we prove a basic lemma.

\begin{lem}
For any continuous map $f\colon X\to X$ and $\Sigma\in\{{\rm C}, {\rm DC1}\}$, $du\Sigma$ implies $gu\Sigma$.
\end{lem}

\begin{proof}
For any $\delta>0$, $i\ge0$, and $n\ge1$, let
\[
S_\delta(i,n)=\{(x,y)\in X^2\colon d(f^i(x),f^i(y))>\delta-\frac{1}{n}\}
\]
and
\[
T(i,n)=\{(x,y)\in X^2\colon d(f^i(x),f^i(y))<\frac{1}{n}\},
\]
open subsets of $X^2$. Then, letting
\[
S_\delta=\bigcap_{n=1}^\infty\bigcap_{j=0}^\infty\bigcup_{i=j}^\infty S_\delta(i,n)\quad\text{and}\quad T=\bigcap_{n=1}^\infty\bigcap_{j=0}^\infty\bigcup_{i=j}^\infty T(i,n),
\]
we obtain
\[
{\rm C}_\delta(X,f)=S_\delta\cap T.
\]
Since $S_\delta$ and $T$ are $G_\delta$-subsets of $X^2$, so is ${\rm C}_\delta(X,f)$. Thus, $du{\rm C}$ implies $gu{\rm C}$.

Similarly, for any $\delta>0$ and $l,m,n\ge1$, let
\[
U_\delta(m,n)=\{(x,y)\in X^2\colon\frac{1}{n}|\{0\le i\le n-1\colon d(f^i(x),f^i(y))<\delta\}|<\frac{1}{m}\}
\]
and
\[
V(l,m,n)=\{(x,y)\in X^2\colon\frac{1}{n}|\{0\le i\le n-1\colon d(f^i(x),f^i(y))<\frac{1}{l}\}|>1-\frac{1}{m}\},
\]
open subsets of $X^2$. Then, letting
\[
U_\delta=\bigcap_{m=1}^\infty\bigcap_{j=1}^\infty\bigcup_{n=j}^\infty U_\delta(m,n)\quad\text{and}\quad V=\bigcap_{l=1}^\infty\bigcap_{m=1}^\infty\bigcap_{j=1}^\infty\bigcup_{n=j}^\infty V(l,m,n),
\]
we obtain
\[
{\rm DC1}_\delta(X,f)=U_\delta\cap V.
\]
Since $U_\delta$ and $V$ are $G_\delta$-subsets of $X^2$, so is ${\rm DC1}_\delta(X,f)$. Thus, $du{\rm DC1}$ implies $gu{\rm DC1}$.
\end{proof}

\begin{rem}
\normalfont
As Example 4.3 shows, $du{\rm DC2}$ does not always imply $g{\rm DC2}$, even if $f$ has the shadowing property.
\end{rem}

We recall a simplified version of Mycielski's theorem: \cite[Theorem 1]{My}.  A topological space is said to be {\em perfect} if it has no isolated point. A subset $S$ of $X$ is called a {\em Mycielski set} if it is a union of countably many Cantor sets.

\begin{lem}
Let $X$ be a perfect complete metric space. If $R_n$ is a residual subset of $X^n$ for each $n\ge2$, then there is a Mycielski set $S$ which is dense in $X$ and satisfies
$(x_1,x_2,\dots,x_n)\in R_n$ for any $n\ge2$ and distinct $x_1,x_2,\dots,x_n\in S$.
\end{lem}

\begin{rem}
\normalfont
Let $f\colon X\to X$ be a continuous map and let $\Sigma\in\{{\rm C},{\rm DC1}\}$. If $f$ exhibits $gu\Sigma$, then $\Sigma_\delta(X,f)$ is a residual subset of $X^2$ for some $\delta>0$, so $X$ is perfect, and by Lemma 1.2, there is a dense Mycielski subset $S$ of $X$ for which any $(x,y)\in S^2$ with $x\ne y$ satisfies $(x,y)\in\Sigma_\delta(X,f)$. Conversely, when $X$ is perfect, if there are $\delta>0$ and a dense subset $S$ of $X$ such that any $(x,y)\in S^2$ with $x\ne y$ satisfies $(x,y)\in\Sigma_\delta(X,f)$, then $\Sigma_\delta(X,f)$ is a dense subset of $X^2$, i.e., $f$ exhibits $du\Sigma$ and so does $gu\Sigma$ by Lemma 1.1. An importance of $gu\Sigma$, $\Sigma\in\{{\rm C},{\rm DC1}\}$, is due to this equivalence. 
\end{rem}

As for the interval maps, in \cite{S1}, Snoha gave some equivalent properties to $gu{\rm C}$ and showed their equivalence to $g{\rm C}$; and in \cite{S2}, some equivalent properties to  $d{\rm C}$ were given. Note that  in \cite{S1}, an example was provided of an interval map which exhibits $d{\rm C}$ but does not exhibit $g{\rm C}$ (see also \cite{R}). As shown by Murinov\'a \cite{Mu}, $gu{\rm C}$ and $g{\rm C}$ are not equivalent for general continuous maps. Compared to these, for all continuous maps with the shadowing property, the main theorems of this paper give some equivalent properties to $gu\Sigma$,  $\Sigma\in\{{\rm C},{\rm DC1}\}$, in terms of chains and show their equivalence to $g\Sigma$. We also provide an example of a continuous map which has the shadowing property and exhibits $d{\rm DC1}$ but does not exhibit $g{\rm C}$.

Any continuous map $f\colon X\to X$ exhibiting $d{\rm C}$ is regionally proximal, so the regional proximality is a common feature of all chaos dealt with in this paper (see Lemma 2.3). In view of the shadowing, this indicates that it is natural to consider a chain analog of the proximality. The precise definitions are given as follows.

\begin{defi}
\normalfont
Given a continuous map $f\colon X\to X$, a pair $(x,y)\in X^2$ is said to be
\begin{itemize}
\item {\em proximal} if
\[
\liminf_{i\to\infty}d(f^i(x),f^i(y))=0,
\]
\item {\em regionally proximal} if for any $\epsilon>0$, there are $(z,w)\in X^2$ and $i\ge0$ such that
\[
\max\{d(x,z),d(y,w),d(f^i(z),f^i(w))\}\le\epsilon,
\]
\item {\em chain proximal} if for any $\delta>0$, there is a pair
\[
((x_i)_{i=0}^k,(y_i)_{i=0}^k)
\]
of $\delta$-chains of $f$ with $(x_0,y_0)=(x,y)$ and $x_k=y_k$.
\end{itemize}

We say that $f$ is {\em proximal} (resp. {\em regionally proximal}, {\em chain proximal}) if every $(x,y)\in X^2$ is a proximal (resp. regionally proximal, chain proximal) pair for $f$.
\end{defi}

\begin{rem}
\normalfont
If $f$ is regionally proximal, then it is chain proximal, and the converse holds if $f$ has the shadowing property. 
\end{rem}

Lemma 3.2 characterizes the chain proximality in terms of the structure of chain components: a continuous map $f\colon X\to X$ is chain proximal if only if it has a unique chain stable chain component, and the map restricted to the component is chain mixing. The definitions of the chain components and the chain stability are given in Section 2.

We see that any continuous map $f\colon X\to X$ exhibiting $du{\rm C}$ is sensitive, so the sensitivity is another common feature of all uniform chaos dealt with in this paper (see Lemma 2.4). Similarly as for the proximality, we define a chain analog of the sensitivity.

\begin{defi}
\normalfont
Let $f\colon X\to X$ be a continuous map. For $e>0$, $x\in X$ is said to be an {\em $e$-sensitive point} for $f$ if for any $\epsilon>0$, there are $y,z\in X$ and $i\ge0$ such that
\[
\max\{d(x,y),d(x,z)\}\le\epsilon
\]
and $d(f^i(y),f^i(z))>e$. We call $x\in X$ a {\em chain $e$-sensitive point} for $f$ if for any $\delta>0$, there is a pair
\[
((x_i)_{i=0}^k,(y_i)_{i=0}^k)
\]
of $\delta$-chains of $f$ with $x_0=y_0=x$ and $d(x_k,y_k)>e$. We say that $f$ is {\em sensitive} (resp. {\em chain sensitive}) if there exists $e>0$ for which every $x\in X$ is an $e$-sensitive (resp. chain $e$-sensitive) point for $f$.
\end{defi}

\begin{rem}
\normalfont
If $f$ is sensitive, then it is chain sensitive, and the converse holds if $f$ has the shadowing property.
\end{rem}

The first main theorem is the following.

\begin{thm}
Given any continuous map $f\colon X\to X$ with the shadowing property, the following properties are equivalent:
\begin{itemize}
\item[(1)] $f$ is chain sensitive and chain proximal,
\item[(2)] $f$ is chain proximal, and the unique chain stable chain component $C$ for $f$ is not a singleton,
\item[(3)] $f$ exhibits $gu{\rm C}$,
\item[(4)] $f$ exhibits $du{\rm C}$,
\item[(5)] $f$ exhibits $g{\rm C}$.
\end{itemize}
\end{thm}

By the properties (1) and (2), for continuous maps with the shadowing property, Theorem 1.1 gives two equivalent properties to $gu{\rm C}$ in terms of chains. Note that by the property (2) with Lemmas 2.1 and 3.2, the global chain structure is largely specified. Also, note that this theorem states the equivalence of $gu{\rm C}$ and $g{\rm C}$ under the shadowing.

To state the other main theorem, we need a definition.

\begin{defi}
\normalfont
Given a continuous map $f\colon X\to X$, a pair $(x,y)\in X^2$ is said to be a {\em distal pair} for $f$ if it is not proximal, i.e.,
\[
\inf_{i\ge0}d(f^i(x),f^i(y))>0.
\]
We say that $f$ satisfies a {\em property S} if there is $e>0$ such that for any $(x,y)\in X^2$ and $\delta>0$, there exists a pair
\[
((x_i)_{i=0}^\infty,(y_i)_{i=0}^\infty)
\]
of $\delta$-pseudo orbits of $f$ with $(x_0,y_0)=(x,y)$ and $\liminf_{i\to\infty}d(x_i,y_i)>e$.
\end{defi}

Then, the theorem is the following.

\begin{thm}
Given any continuous map $f\colon X\to X$ with the shadowing property, the following properties are equivalent:
\begin{itemize}
\item[(1)] $f$ has the property S and is chain proximal,
\item[(2)] $f$ is chain proximal, and the unique chain stable chain component $C$ for $f$ contains a distal pair for $f|_C$,
\item[(3)] $f$ exhibits $gu{\rm DC1}$,
\item[(4)] $f$ exhibits $du{\rm DC1}$,
\item[(5)] $f$ exhibits $g{\rm DC1}$.
\end{itemize}
\end{thm}

For all continuous maps with the shadowing property, the properties (1) and (2) give two equivalent properties to $gu{\rm DC1}$ in terms of chains. Similarly as in the case of Theorem 1.1, by the property (2) with Lemmas 2.1 and 3.2, the global chain structure is largely determined. Also, Theorem 1.2 states the equivalence of $gu{\rm DC1}$ and $g{\rm DC1}$ under the shadowing.

This paper consists of four sections. In Section 2, we give some basic definitions and prove preparatory lemmas. The main theorems are proved in Section 3. In Section 4, we give several examples illustrating the main theorems. Example 4.1 gives a very simple proximal continuous map exhibiting $gu{\rm C}$. Example 4.2 is a modification of Example 4.1. It shows that $gu{\rm C}$ and the property (2) in Theorem 1.2 do not necessarily imply $gu{\rm DC1}$ in absence of the shadowing property. Example 4.3 gives a continuous map $f\colon X\to X$ with the following properties:
\begin{itemize}
\item $f$ has the shadowing property,
\item $f$ exhibits $gu{\rm C}$,
\item $f$ has the limit-shadowing property, but its restriction to the chain recurrent set does not have the limit-shadowing property,   
\item $f$ exhibits $du{\rm DC2}$ and $d{\rm DC1}$ but does not exhibit $g{\rm DC2}$.
\end{itemize}
The third property gives an answer to a question in \cite{K1}. Example 4.4 provides a continuous map with the shadowing property exhibiting $gu{\rm DC1}$. Finally, Example 4.5 gives a continuous map which has the shadowing property and exhibits $d{\rm DC1}$ but does not exhibit $g{\rm C}$.

\section{Basic definitions and lemmas}

In this section, we define the chain components and the chain stability, and prove some lemmas. 

\begin{defi}
\normalfont
Given a continuous map $f\colon X\to X$, a $\delta$-chain $(x_i)_{i=0}^{k}$ of $f$ is said to be a {\em $\delta$-cycle} of $f$ if $x_0=x_k$.  We call $x\in X$ a {\em chain recurrent point} for $f$ if for any $\delta>0$, there exists a $\delta$-cycle $(x_i)_{i=0}^{k}$ of $f$ with $x_0=x_k=x$.
\end{defi}

The set of chain recurrent points for $f$ is a closed $f$-invariant subset of $X$, and we denote it by $CR(f)$. For any $x,y\in X$ and $\delta>0$, the notation $x\rightarrow_\delta y$ means that there is a $\delta$-chain $(x_i)_{i=0}^k$ of $f$ with $(x_0,x_k)=(x,y)$. For $x,y\in X$, we write $x\rightarrow y$ if $x\rightarrow_\delta y$ for all $\delta>0$.

\begin{defi}
\normalfont
A continuous map $f\colon X\to X$ is said to be {\em chain transitive} if $x\rightarrow y$ for any $x,y\in X$.
\end{defi}

For any $\delta>0$, define an equivalence relation $\leftrightarrow_\delta$ in
\[
CR(f)^2=CR(f)\times CR(f)
\]
as follows: for any $x,y\in CR(f)$, $x\leftrightarrow_\delta y$ iff  $x\rightarrow_\delta y$ and $y\rightarrow_\delta x$. We denote by $\mathcal{C}_\delta(f)$ the set of equivalence classes of $\leftrightarrow_\delta$. Note that every $x\in CR(f)$ satisfies $x\leftrightarrow_\delta f(x)$, and any $x,y\in CR(f)$ with $d(x,y)\le\delta$ satisfy $x\leftrightarrow_\delta y$, so each $C\in\mathcal{C}_\delta(f)$ is a clopen $f$-invariant subset of $CR(f)$. We define an equivalence relation $\leftrightarrow$ in $CR(f)^2$ by for any $x,y\in CR(f)$, $x\leftrightarrow y$ iff $x\rightarrow y$ and $y\rightarrow x$, which is equivalent to $x\leftrightarrow_\delta y$ for all $\delta>0$.

\begin{defi}
\normalfont
An equivalence class $C$ of $\leftrightarrow$ is said to be a {\em chain component} for $f$. We denote by $\mathcal{C}(f)$ the set of chain components for $f$.
\end{defi}

Then, the following properties hold:
\begin{itemize}
\item $CR(f)=\bigsqcup_{C\in\mathcal{C}(f)}C,$
\item Any $C\in\mathcal{C}(f)$ is a closed $f$-invariant subset of $CR(f)$,
\item $f|_C\colon C\to C$ is chain transitive for all $C\in\mathcal{C}(f)$.
\end{itemize}

A partition $\mathcal{P}$ of a set $Y$ is a set of disjoint subsets of $Y$ whose union is $Y$. For any partition $\mathcal{P}$ of $Y$ and $y\in Y$, we define $\mathcal{P}(y)\in\mathcal{P}$ by $y\in\mathcal{P}(y)$. Note that for any continuous map $f\colon X\to X$, $\mathcal{C}_\delta(f)$, $\delta>0$, and $\mathcal{C}(f)$ are partitions of $CR(f)$, and
\[
[\mathcal{C}(f)](x)=\bigcap_{\delta>0}[\mathcal{C}_\delta(f)](x)
\]   
for every $x\in CR(f)$.

\begin{defi}
\normalfont
Given a continuous map $f\colon X\to X$, we say that a closed $f$-invariant subset $S$ of $X$ is {\em chain stable} if for any $\epsilon>0$, there is $\delta>0$ such that every $\delta$-chain $(x_i)_{i=0}^k$ of $f$ with $x_0\in S$ satisfies $d(x_i,S)\le\epsilon$ for all $0\le i\le k$. We denote by $\mathcal{C}^s(f)$ the set of $C\in\mathcal{C}(f)$ being chain stable.
\end{defi}

The following lemma states that for any continuous map $f\colon X\to X$, any $x\in X$ can reach to a single chain stable chain component for $f$ by arbitrarily precise chains of $f$. It describes a global structure of chains. 

\begin{lem}
Let $f\colon X\to X$ be a continuous map. For any $x\in X$, there exists $C\in\mathcal{C}^s(f)$ such that for every $\delta>0$, there is a $\delta$-chain $(x_i)_{i=0}^k$ of $f$ with $x_0=x$ and $x_k\in C$.
\end{lem}

\begin{proof}
For every $\delta>0$, define an order $\le_\delta$ on $\mathcal{C}_\delta(f)$ by for any $A,B\in\mathcal{C}_\delta(f)$, $A\le_\delta B$ iff there is $(x,y)\in A\times B$ such that $x\rightarrow_\delta y$. Also, define an order $\le$ on $\mathcal{C}(f)$ by for any $E,F\in\mathcal{C}(f)$, $E\le F$ iff there is $(x,y)\in E\times F$ such that $x\rightarrow y$. We show that every chain $R$ in $\mathcal{C}(f)$ has an upper bound in $\mathcal{C}(f)$ (with respect to $\le$).

Given any chain $R$ in $\mathcal{C}(f)$ and $\delta>0$, let
\[
R_\delta=\{A\in\mathcal{C}_\delta(f)\colon\exists E\in R\:\:\text{s.t.}\:\: E\subset A\}.
\]
Then, $R_\delta$ is a chain in $\mathcal{C}_\delta(f)$ with respect to $\le_\delta$. Since $\mathcal{C}_\delta(f)$ is a finite set and so is $R_\delta$, there is a maximal element $A_\delta$ of $R_\delta$ with respect to $\le_\delta$. Choose a sequence $0<\delta_1>\delta_2>\cdots\to0$ and a closed subset $Y$ of $X$ with
\[
\lim_{n\to\infty}d_H(A_{\delta_n},Y)=0,
\]
where $d_H$ is the Hausdorff distance. Since $A_{\delta_n}\subset CR(f)$ for all $n\ge1$, we have $Y\subset CR(f)$. Since $x\rightarrow_{\delta_n}y$ for any $n\ge1$ and $x,y\in A_{\delta_n}$, we easily see that $x\rightarrow y$ for all $x,y\in Y$, so there is $E\in\mathcal{C}(f)$ with $Y\subset E$. Given any $F\in R$ and  $n\ge1$, we have $F\subset B_n$ for some $B_n\in R_{\delta_n}$. By the maximality of $A_{\delta_n}$ in $R_{\delta_n}$, we obtain $B_n\le_{\delta_n}A_{\delta_n}$, so there is $(x_n,y_n)\in B_n\times A_{\delta_n}$ with $x_n\rightarrow_{\delta_n}y_n$. Note that $B_1\supset B_2\supset\cdots$ and
\[
F=\bigcap_{n\ge1}B_n.
\]
It follows that $x\rightarrow y$ for some
\[
(x,y)\in F\times Y\subset F\times E,
\]
so $F\le E$. Since $F\in R$ is arbitrary, $E$ is the desired upper bound of $R$.

Let us continue the proof. For any $x\in X$, there is $y\in CR(f)$ such that $x\rightarrow y$. Take $D\in\mathcal{C}(f)$ with $y\in D$ and consider a subset
\[
\mathcal{S}=\{F\in\mathcal{C}(f)\colon D\le F\}
\]
of $\mathcal{C}(f)$. By what is shown above, Zorn's lemma gives a maximal element $C$ of $\mathcal{S}$ with respect to $\le$. Since $D\le C$, we easily see that for every $\delta>0$, there is a $\delta$-chain $(x_i)_{i=0}^k$ of $f$ with $x_0=x$ and $x_k\in C$. What is left to show is that $C$ is chain stable. If not, then by the compactness of $X$, we have $z\rightarrow w$ for some
\[
(z,w)\in C\times[X\setminus C].
\]
By taking $F\in\mathcal{C}(f)$ with $\omega(w,f)\subset F$, here $\omega(\cdot,f)$ is the $\omega$-limit set, we obtain $C\le F$. If $C\ne F$, this contradicts the maximality of $C$ in $\mathcal{S}$. If $C=F$, we obtain $w\in C$, which is also a contradiction. Thus, $C$ is chain stable, and so the lemma has been proved.
\end{proof}

\begin{cor}
For any continuous map $f\colon X\to X$, $\mathcal{C}^s(f)$ is not the empty set.
\end{cor}

Let $f\colon X\to X$ be continuous map. For a closed $f$-invariant subset $S$ of $X$, we say that $S$ is {\em Lyapunov stable} if for any $\epsilon>0$, there is $\delta>0$ such that every $x\in X$ with $d(x,S)\le\delta$ satisfies $d(f^i(x),S)\le\epsilon$ for all $i\ge0$. If $S$ is chain stable, then it is Lyapunov stable. The stable set $W^s(S,f)$ is defined by
\[
W^s(S,f)=\{x\in X\colon \lim_{i\to\infty}d(f^i(x),S)=0\}.
\]

\begin{lem}
Let $f\colon X\to X$ be continuous map and let $S$ be a closed $f$-invariant subset of $X$. If $S$ is Lyapunov stable, then $W^s(S,f)$ is a $G_\delta$-subset of $X$. 
\end{lem}

\begin{proof}
For any $i\ge0$ and $n\ge1$, let
\[
S(i,n)=\{x\in X\colon d(f^i(x),S)<\frac{1}{n}\},
\] 
an open subset of $X$. Since $S$ is Lyapunov stable, we have
\[
W^s(S,f)=\bigcap_{n=1}^\infty\bigcup_{i=0}^\infty S(i,n),
\]
which is a $G_\delta$-subset of $X$.
\end{proof}

The following two lemmas are mentioned in Section 1.

\begin{lem}
Let $f\colon X\to X$ be a continuous map. If $f$ exhibits $d{\rm C}$, then $f$ is regionally proximal and so is chain proximal.
\end{lem}

\begin{proof}
Since $f$ exhibits $d{\rm C}$, for any $(x,y)\in X^2$ and $\epsilon>0$, there is $(z,w)\in{\rm C}(X,f)$ such that
\[
\max\{d(x,z),d(y,w)\}\le\epsilon.
\]
By the definition of ${\rm C}(X,f)$, we have
\[
\liminf_{i\to\infty}d(f^i(z),f^i(w))=0,
\]
i.e., $(z,w)$ is a proximal pair for $f$, so there is $i\ge0$ with $d(f^i(z),f^i(w))\le\epsilon$. These properties imply that $f$ is regionally proximal.
\end{proof}

\begin{lem}
If a continuous map $f\colon X\to X$ exhibits $du{\rm C}$, then $f$ is sensitive.
\end{lem}

\begin{proof}
By the definition of $du{\rm C}$, there is $\delta>0$ such that ${\rm C}_\delta(X,f)$ is a dense subset of $X^2$. Take $0<e<\delta$. Then, for any $x\in X$ and $\epsilon>0$, there is $(y,z)\in{\rm C}_\delta(X,f)$ such that
\[
\max\{d(x,y),d(x,z)\}\le\epsilon.
\]
Since $(y,z)\in{\rm C}_\delta(X,f)$, we have
\[
\limsup_{i\to\infty}d(f^i(y),f^i(z))\ge\delta>e,
\]
so there is $i\ge0$ with $d(f^i(y),f^i(z))>e$. This implies that $x$ is an $e$-sensitive point for $f$. Since $x\in X$ is arbitrary, $f$ is sensitive.
\end{proof}

\section{Proof of the main theorems}

In this section, we prove Theorems 1.1 and 1.2. Let us begin with a definition.

\begin{defi}
\normalfont
A continuous map $f\colon X\to X$ is said to be {\em chain mixing} if for any $\delta>0$, there exists $N>0$ such that for any $x,y\in X$ and $k\ge N$, there is a $\delta$-chain $(x_i)_{i=0}^k$ of $f$ with $(x_0,x_k)=(x,y)$. We say that $f$ is {\em mixing} if for any non-empty open subsets $U$, $V$ of $X$, there is $N>0$ such that $f^i(U)\cap V\ne\emptyset$ holds for any $i\ge N$.
\end{defi}

\begin{rem}
\normalfont
If $f$ is mixing, then it is chain mixing, and the converse holds if $f$ has the shadowing property.
\end{rem}

As shown in \cite{RW}, the following holds.

\begin{lem}
Let $f\colon X \to X$ be a continuous map. If $f$ is chain transitive, then the following properties are equivalent:
\begin{itemize}
\item[(1)] $f$ is chain mixing,
\item[(2)] The product map $f\times f\colon X\times X\to X\times X$ is chain transitive,
\item[(3)] $f$ is chain proximal.
\end{itemize}
\end{lem}

Recall that for any continuous map $f\colon X\to X$, $\mathcal{C}^s(f)$ denotes the set of chain stable $C\in\mathcal{C}(f)$. Then, the following lemma characterizes the chain proximality in terms of the chain components. Combined with Lemma 2.1, it describes a global structure of chains for chain proximal maps. 

\begin{lem}
For any continuous map $f\colon X\to X$, the following properties are equivalent:
\begin{itemize}
\item[(1)] $f$ is chain proximal,
\item[(2)] $\mathcal{C}^s(f)$ is a singleton, and $f|_{C}\colon C\to C$ is chain mixing for the unique $C\in\mathcal{C}^s(f)$.
\end{itemize}
\end{lem}

\begin{proof}
The implication $(1)\Rightarrow(2)$: By Corollary 2.1, there is one or more chain stable chain components for $f$. Suppose that $C,D\in C(f)$ with $C\ne D$ are chain stable, and fix $(x,y)\in C\times D$. Then, there are $\epsilon>0$ and $\delta>0$ such that
\[
\inf\{d(z,w)\colon z\in C,w\in D\}>3\epsilon,
\]
and any pair
\[
((x_i)_{i=0}^k,(y_i)_{i=0}^k)
\]
of $\delta$-chains of $f$ with $(x_0,y_0)=(x,y)$ satisfies
\[
\max\{d(x_k,C),d(y_k,D)\}\le\epsilon,
\]
so $d(x_k,y_k)>\epsilon$. This contradicts that $f$ is chain proximal; therefore, there is a unique chain stable chain component $C\in\mathcal{C}(f)$, i.e., $\mathcal{C}^s(f)=\{C\}$, a singleton. By the chain proximality of $f$ and the chain stability of $C$, we easily see that $f|_C$ is chain proximal. Since $f|_C$ is chain transitive, by Lemma 3.1, it is chain mixing.
$ $\newline
$(2)\Rightarrow(1)$: Since $\mathcal{C}^s(f)=\{C\}$, for any $(x,y)\in X^2$ and $\delta>0$, by Lemma 2.1, there is a pair
\[
(\alpha_1,\beta_1)=((x_i)_{i=0}^k,(y_i)_{i=0}^k)
\]
of $\delta$-chains of $f$ with $(x_0,y_0)=(x,y)$ and $\{x_k,y_k\}\subset C$. Since $f|_C$ is chain mixing, for a sufficiently large $l>0$, we have a pair
\[
(\alpha_2,\beta_2)=((z_i)_{i=0}^l,(w_i)_{i=0}^l)
\]
of $\delta$-chains of $f$ with $(z_0,w_0)=(x_k,y_k)$ and $z_l=w_l$. By concatenating them, we obtain a pair
\[
(\alpha_3,\beta_3)=(\alpha_1\alpha_2,\beta_1\beta_2)=((u_i)_{i=0}^{k+l},(v_i)_{i=0}^{k+l})
\]
of $\delta$-chains of $f$ with $(u_0,v_0)=(x,y)$ and $u_{k+l}=v_{k+l}$. This implies that $(x,y)$ is a proximal pair for $f$. Since $(x,y)\in X^2$ is arbitrary, $f$ is chain proximal.
\end{proof}

The following lemma states that the properties (1) and (2) in Theorem 1.1 are equivalent for general continuous maps.

\begin{lem}
For any continuous map $f\colon X\to X$, the following properties are equivalent:
\begin{itemize}
\item[(1)] $f$ is chain sensitive and chain proximal,
\item[(2)] $f$ is chain proximal, and the unique $C\in\mathcal{C}^s(f)$ is not a singleton.
\end{itemize}
\end{lem}

\begin{proof}
The implication $(1)\Rightarrow(2)$: Since $f$ is chain proximal, by Lemma 3.2, there is a unique $C\in\mathcal{C}^s(f)$. If $C=\{x\}$, a singleton, then for any $\epsilon>0$, the chain stability of $C$ gives $\delta>0$ for which every $\delta$-chain $(x_i)_{i=0}^k$ of $f$ with $x_0=x$ satisfies $d(x_i,x)\le\epsilon$ for all $0\le i\le k$, implying that $x$ is not a chain $2\epsilon$-sensitive point for $f$. This contradicts that $f$ is chain sensitive. Thus, $C$ is not a singleton.   
$ $\newline
$(2)\Rightarrow(1)$:
Fix $0<e<{\rm diam}\:C$. For any $x\in X$ and $\delta>0$, Lemma 2.1 gives a $\delta$-chain $\alpha=(x_i)_{i=0}^k$ of $f$ with $x_0=x$ and $x_k\in C$. By Lemma 3.2, $f|_C$ is chain mixing, so for any sufficiently large $l>0$, we have a pair
\[
(\beta_1,\beta_2)=((y_i)_{i=0}^l,(z_i)_{i=0}^l)
\]
of $\delta$-chains of $f|_C$ with $y_0=z_0=x_k$ and $d(y_l,z_l)>e$. Then, a pair
\[
(\alpha_1,\alpha_2)=(\alpha\beta_1,\alpha\beta_2)=((u_i)_{i=0}^{k+l},(v_i)_{i=0}^{k+l})
\]
of $\delta$-chains of $f$ satisfies $u_0=v_0=x$ and $d(u_{k+l},v_{k+l})>e$. This implies that $x$ is a chain $e$-sensitive point for $f$. Since $x\in X$ is arbitrary, $f$ is chain sensitive.
\end{proof}

Although the next lemma is shown in \cite{Mu}, we give a proof of it for completeness.

\begin{lem}
If a continuous map $f\colon X\to X$ is sensitive and regionally proximal, then $f$ exhibits $gu{\rm C}$.
\end{lem}

\begin{proof}
As in the proof of Lemma 1.1, for $\delta>0$, we express ${\rm C}_\delta(X,f)$ as $S_\delta\cap T$ where
\[
S_\delta=\bigcap_{n=1}^\infty\bigcap_{j=0}^\infty\bigcup_{i=j}^\infty S_\delta(i,n)\quad\text{and}\quad T=\bigcap_{n=1}^\infty\bigcap_{j=0}^\infty\bigcup_{i=j}^\infty T(i,n).
\]
Take $e>0$ for which every $x\in X$ is an $e$-sensitive point for $f$. Then, we easily see that for any $j\ge0$, $(x,y)\in X^2$, and $\epsilon>0$, there are $(z,w)\in X^2$ and $i\ge j$ such that
\[
\max\{d(x,z),d(y,w)\}\le\epsilon
\]
and $d(f^i(z),f^i(w))>e/2$. This implies that for any $n\ge 1$, $\bigcup_{i=j}^\infty S_{e/2}(i,n)$, $j\ge0$, are open and dense subsets of $X^2$, and so $S_{e/2}$ is a dense $G_\delta$-subset of $X^2$. On the other hand, since $f$ is regionally proximal,  for any $j\ge0$, $(x,y)\in X^2$, and $\epsilon>0$, there are $(z,w)\in X^2$ and $i\ge j$ such that
\[
\max\{d(x,z),d(y,w),d(f^i(z),f^i(w))\}\le\epsilon.
\]
It follows that for any $n\ge1$, $\bigcup_{i=j}^\infty T(i,n)$, $j\ge0$, are open and dense subsets of $X^2$, and so $T$ is a dense $G_\delta$-subset of $X^2$.
Thus, ${\rm C}_{e/2}(X,f)$ is a dense $G_\delta$-subset of $X^2$, proving that $f$ exhibits $guC$.
\end{proof}

Under the assumption of the shadowing property, the next lemma extracts from $g{\rm C}$ a property of  the unique chain stable chain component: not a singleton. 

\begin{lem}
Let $f\colon X\to X$ be a continuous map with shadowing property. If $f$ exhibits $g{\rm C}$, then $f$ is chain proximal, and the unique $C\in\mathcal{C}^s(f)$ is not a singleton.
\end{lem}

\begin{proof}
Since $g{\rm C}$ implies $d{\rm C}$, by Lemma 2.3, $f$ is chain proximal. It remains to prove that the unique $C\in\mathcal{C}^s(f)$ is not a singleton. By the definition of $g{\rm C}$, ${\rm C}(X,f)$ is a residual subset of $X^2$. Since $C$ is chain stable and so Lyapunov stable, as in the proof of  Lemma 2.2, we have
\[
W^s(C,f)=\bigcap_{n=1}^\infty\bigcup_{i=0}^\infty S(i,n),
\]
where 
\[
S(i,n)=\{x\in X\colon d(f^i(x),C)<\frac{1}{n}\}
\] 
for all $i\ge0$ and $n\ge1$. Since $f$ has the shadowing property, for any $\epsilon>0$, there is $\delta>0$ such that every $\delta$-pseudo orbit of $f$ is $\epsilon$-shadowed by some point of $X$. Then, for any $x\in X$, Lemma 2.1 gives a $\delta$-chain $(x_i)_{i=0}^k$ of $f$ with $x_0=x$ and $x_k\in C$. Consider a $\delta$-pseudo orbit
\[
\xi=(x_0,x_1,\dots,x_k,f(x_k),f^2(x_k),\dots)
\] 
of $f$, which is $\epsilon$-shadowed by $y\in X$, and note that $y$ satisfies
\[
\max\{d(x,y),d(f^k(y),C)\}\le\epsilon.
\]
Since $\epsilon>0$ is arbitrary, these properties clearly imply that $\bigcup_{i=0}^\infty S(i,n)$, $n\ge1$, are open and dense subsets of $X$, so $W^s(C,f)$ is a dense $G_\delta$-subset of $X$. It follows that $W^s(C,f)\times W^s(C,f)$ is a dense $G_\delta$-subset of $X^2$, so
\[
{\rm C}(X,f)\cap[W^s(C,f)\times W^s(C,f)]
\]
is a residual subset of $X^2$. Take any
\[
(z,w)\in{\rm C}(X,f)\cap[W^s(C,f)\times W^s(C,f)].
\]
Since $(z,w)\in{\rm C}(X,f)$, we have
\[
\limsup_{i\to\infty}d(f^i(z),f^i(w))\ge\delta
\]
for some $\delta>0$, so there are a sequence $0\le i_1<i_2<\cdots$ and $(a,b)\in X^2$ such that
\[
\limsup_{j\to\infty}d(f^{i_j}(z),f^{i_j}(w))\ge\delta
\]
and $\lim_{j\to\infty}(f^{i_j}(z),f^{i_j}(w))=(a,b)$. Then, $d(a,b)\ge\delta$ and so $a\ne b$. By $(z,w)\in W^s(C,f)\times W^s(C,f)$, we obtain $(a,b)\in C\times C$; therefore, $C$ is not a singleton, completing the proof.
\end{proof}

Let us complete the proof of Theorem 1.1. 

\begin{proof}[Proof of Theorem 1.1]
By Lemma 3.3, (1) and (2) are equivalent. Due to Lemma 1.1, (3) and (4) are equivalent. By Lemma 3.4, (1) implies (3). The property (3) implies (5), and by Lemma 3.5, (5) implies (2); therefore, the theorem has been proved.
\end{proof}
 
We proceed to the proof of Theorem 1.2. The first lemma states the equivalence of the properties (1) and (2) in Theorem 1.2 for general continuous maps.

\begin{lem}
For any continuous map $f\colon X\to X$, the following properties are equivalent:
\begin{itemize}
\item[(1)] $f$ satisfies the property S and is chain proximal,
\item[(2)] $f$ is chain proximal, and the unique $C\in\mathcal{C}^s(f)$ contains a distal pair for $f|_C$.
\end{itemize}
\end{lem}

\begin{proof}
The implication $(1)\Rightarrow(2)$: Since $f$ is chain proximal, by Lemma 3.2, there is a unique $C\in\mathcal{C}^s(f)$. If $f|_C$ is proximal, then for any $\epsilon>0$, there is $N>0$ such that
\[
\inf_{0\le i\le N}d(f^i(x),f^i(y))\le\epsilon
\]
for all $(x,y)\in C^2$. Then, for a sufficiently small $\delta>0$, we have
\[
\inf_{0\le i\le N}d(z_i,w_i)\le2\epsilon
\]
for any pair 
\[
((z_i)_{i=0}^N,(w_i)_{i=0}^N)
\]
of $\delta$-chains of $f|_{C}$. Fix any $(x,y)\in C^2$. Since $C$ is chain stable, if $\eta>0$ is sufficiently small, then for every pair
\[
((x_i)_{i=0}^\infty,(y_i)_{i=0}^\infty)
\]
of $\eta$-pseudo orbits of $f$ with $(x_0,y_0)=(x,y)$, by taking $(z_i,w_i)\in C^2$ with $(z_0,w_0)=(x,y)$, $d(x_i,z_i)=d(x_i,C)\le\epsilon$, and $d(y_i,w_i)=d(y_i,C)\le\epsilon$ for all $i>0$, we obtain a pair 
\[
((z_i)_{i=0}^\infty,(w_i)_{i=0}^\infty)
\]
of $\delta$-pseudo orbits of $f|_C$. Then, we see that
\[
\liminf_{i\to\infty}d(x_i,y_i)\le\liminf_{i\to\infty}d(z_i,w_i)+2\epsilon\le4\epsilon.
\]
This contradicts that $f$ satisfies the property S. Thus, $f|_{C}\colon C\to C$ is not proximal, or equivalently, $C$ contains a distal pair for $f|_C$. 
$ $\newline
$(2)\Rightarrow(1)$:
Fix a distal pair $(z,w)\in C^2$ for $f|_C$ and $e>0$ with
\[
\inf_{i\ge0}d(f^i(z),f^i(w))>e.
\]
For any $(x,y)\in X^2$ and $\delta>0$, Lemma 2.1 gives a
pair
\[
(\alpha_1,\beta_1)=((x_i)_{i=0}^k,(y_i)_{i=0}^k)
\]
of $\delta$-chains of $f$ with $(x_0,y_0)=(x,y)$ and $\{x_k,y_k\}\subset C$. By Lemma 3.2, $f|_C$ is chain mixing, so for a sufficiently large $l>0$, we have a pair
\[
(\alpha_2,\beta_2)=((z_i)_{i=0}^l,(w_i)_{i=0}^l)
\]
of $\delta$-chains of $f|_C$ with $(z_0,w_0)=(x_k,y_k)$ and $(z_l,w_l)=(z,w)$. Put
\[
(\alpha_3,\beta_3)=((z,f(z),f^2(z),\dots),(w,f(w),f^2(w),\dots))
\]
and consider a pair
\[
(\alpha,\beta)=(\alpha_1\alpha_2\alpha_3,\beta_1\beta_2\beta_3)=((u_i)_{i=0}^\infty,(v_i)_{i=0}^\infty)
\]
of $\delta$-pseudo orbits of $f$. Then, we have $(u_0,v_0)=(x,y)$ and
\[
\liminf_{i\to\infty}d(u_i,v_i)\ge\inf_{i\ge0}d(f^i(z),f^i(w))>e.
\]
Since $(x,y)\in X^2$ and $\delta>0$ are arbitrary, we conclude that $f$ satisfies the property S.
\end{proof}

The next lemma concerns the implication $(1)\Rightarrow(3)$ in Theorem 1.2.

\begin{lem}
Let $f\colon X\to X$  be a continuous map with the shadowing property. If $f$ has the property S and is chain proximal, then $f$ exhibits $gu{\rm DC1}$.
\end{lem}

\begin{proof}
We apply an argument in \cite{LLT}. As in the proof of Lemma 1.1, for $\delta>0$, express ${\rm DC1}_\delta(X,f)$ as $U_\delta\cap V$ where
\[
U_\delta=\bigcap_{m=1}^\infty\bigcap_{j=1}^\infty\bigcup_{n=j}^\infty U_\delta(m,n)\quad\text{and}\quad V=\bigcap_{l=1}^\infty\bigcap_{m=1}^\infty\bigcap_{j=1}^\infty\bigcup_{n=j}^\infty V(l,m,n).
\]
Take $e>0$ as in the definition of property S and fix $0<\delta<e$. Put
\[
D^\delta(X,f)=\{(x,y)\in X^2\colon\liminf_{i\to\infty}d(f^i(x),f^i(y))>\delta\},
\]
and
\[
A^\sigma(X,f)=\{(x,y)\in X^2\colon\limsup_{i\to\infty}d(f^i(x),f^i(y))<\sigma\}
\]
for $\sigma>0$. Note that
\[
D^\delta(X,f)\subset\cup_{n=j}^\infty U_\delta(m,n)
\quad
\text{and}
\quad
A^{\frac{1}{l}}(X,f)\subset\cup_{n=j}^\infty V(l,m,n)
\]
for all $l,m,j\ge1$.

Since $f$ has the shadowing property, by the property S of $f$, $D^\delta(X,f)$ is dense in $X^2$, so $U_\delta$ is a dense $G_\delta$-subset of $X^2$. Similarly, by the shadowing property and the chain proximality of  $f$, $A^{\frac{1}{l}}(X,f)$ is dense in $X^2$ for all $l\ge1$, so $V$ is a dense $G_\delta$-subset of $X^2$. Thus, ${\rm DC1}_\delta(X,f)$ is a dense $G_\delta$-subset of $X^2$, proving that $f$ exhibits $gu{\rm DC1}$.
\end{proof}

Under the assumption of the shadowing property, the next lemma extracts from $g{\rm DC1}$ a property of  the unique chain stable chain component: the existence of a distal pair.

\begin{lem}
Let $f\colon X\to X$ be a continuous map with shadowing property. If $f$ exhibits $g{\rm DC1}$, then $f$ is chain proximal, and the unique $C\in\mathcal{C}^s(f)$ contains a distal pair for $f|_C$.
\end{lem}

\begin{proof}
Since $g{\rm DC1}$ implies $d{\rm C}$, by Lemma 2.3, $f$ is chain proximal. It remains to prove that the unique $C\in\mathcal{C}^s(f)$ contains a distal pair for $f|_C$. By the definition of $g{\rm DC1}$, ${\rm DC1}(X,f)$ is a residual subset of $X^2$. Similarly as in the proof of Lemma 3.5, $W^s(C,f)\times W^s(C,f)$ is a dense $G_\delta$-subset of $X^2$, so
\[
{\rm DC1}(X,f)\cap[W^s(C,f)\times W^s(C,f)]
\]
is a residual subset of $X^2$. Take any
\[
(z,w)\in{\rm DC1}(X,f)\cap[W^s(C,f)\times W^s(C,f)].
\]
Since $(z,w)\in{\rm DC1}(X,f)$, we have $F_{zw}(\delta)=0$ for some $\delta>0$, so there are sequences $0\le i_1<i_2<\cdots$, $0<N_1<N_2<\cdots$, and $(a,b)\in X^2$ such that
\[
\inf_{i_j\le i\le i_j+N_j}d(f^i(z),f^i(w))\ge\delta
\]
for all $j\ge1$, and $\lim_{j\to\infty}(f^{i_j}(z),f^{i_j}(w))=(a,b)$. If follows that
\[
\inf_{i\ge0}d(f^i(a),f^i(b))\ge\delta.
\]
By $(z,w)\in W^s(C,f)\times W^s(C,f)$, we obtain $(a,b)\in C\times C$; therefore, $C$ contains a distal pair for $f|_C$, proving the lemma. 
\end{proof}

As the final proof of this section, we complete the proof of Theorem 1.2.

\begin{proof}[Proof of Theorem 1.2]
By Lemma 3.6, (1) and  (2) are equivalent. Due to Lemma 1.1, (3) and (4) are equivalent. By Lemma 3.7, (1) implies (3). The property (3) implies (5), and by Lemma 3.8, (5) implies (2), thus the theorem has been proved.
\end{proof}

\section{Examples}

In this section, we give five examples illustrating the main theorems. We begin with a definition.

\begin{defi}
\normalfont
Let $f\colon X\to X$ be a continuous map and let $\xi=(x_i)_{i\ge0}$ be a sequence of points in $X$. We call $\xi$ a {\em limit-pseudo orbit} of $f$ if
\[
\lim_{i\to\infty}d(f(x_i),x_{i+1})=0.
\]
We say that $\xi$ is {\em limit-shadowed} by $x\in X$ if
\[
\lim_{i\to\infty}d(f^i(x),x_i)=0.
\]
Then, $f$ is said to have the {\em limit-shadowing property} if every limit-pseudo orbit $\xi$ of $f$ is limit-shadowed by some $x\in X$. 
\end{defi}

\begin{rem}
\normalfont
For example, any subshift of finite type (over finite alphabets) is positively expansive and has the shadowing property, so has the limit-shadowing property (see \cite{AH,BGO}).
\end{rem}

\begin{lem}
Let $X=\{0,1\}^\mathbb{N}$ and let $\sigma\colon X\to X$ be the shift map. Let
\[
X_\infty=\{0^\infty,10^\infty\}\cup\{0^m10^\infty\colon m\ge1\},
\]
a closed $\sigma$-invariant subset of $X$. For any $x\in X$, if $x\in W^s(X_\infty,\sigma)$, then
\[
\lim_{n\to\infty}\frac{1}{n}|\{0\le i\le n-1\colon d(\sigma^i(x),0^\infty)<\epsilon\}|=1
\]
for all $\epsilon>0$.
\end{lem}

\begin{proof}
Given any $x=(x_i)_{i\ge1}\in X$ with $\lim_{i\to\infty}d(\sigma^i(x),X_\infty)=0$, let
\[
S=\{i\ge1\colon x_i=1\}.
\]
If $S$ is a finite set, then $\sigma^i(x)=0^\infty$ for any sufficiently large $i\ge0$, so the conclusion of the lemma clearly holds. Suppose that $S$ is an infinite set, and let
\[
S=\{i_j\colon j\ge1\},
\]
where $i_1<i_2<\cdots$. Then, we see that $i_{j+1}-i_j\to\infty$ as $j\to\infty$, so for any $k\ge1$,
\[
\lim_{n\to\infty}\frac{1}{n}|\{0\le i\le n-1\colon x_{i+1}x_{i+2}\cdots x_{i+k}=0^k\}|=1.
\]
This clearly implies the conclusion of the lemma. 
\end{proof}

The first example gives a simple proximal continuous map which exhibits $gu{\rm C}$ but has no DC2-$\delta$-scrambled pair for any $\delta>0$.

\begin{ex}
\normalfont
Let $X=\{0,1\}^\mathbb{N}$ and $\sigma\colon X\to X$ be the shift map. We put
\begin{equation*}
\hat{a}=
\begin{cases}
0&\text{if $a=0$}\\
10&\text{if $a=1$}
\end{cases}
\end{equation*}
and define $\pi\colon X\to X$ by $\pi(x)=\hat{x_1}\hat{x_2}\hat{x_3}\cdots$ for all $x=(x_i)_{i\ge1}\in X$. Put $f=\sigma\circ\pi\colon X\to X$ and
\[
X_\infty=\{0^\infty,10^\infty\}\cup\{0^m10^\infty\colon m\ge1\}\subset X.
\]

We have $\sigma(X_\infty)=f(X_\infty)=X_\infty$ and $\sigma|_{X_\infty}=f|_{X_\infty}$. By this, it is easy to see that $X_\infty\subset CR(f)$. Since $CR(f)\subset\bigcap_{i\ge0}f^i(X)$ and $\bigcap_{i\ge0}f^i(X)\subset X_\infty$, we obtain $CR(f)=X_\infty$. If $f$ has a distal pair, then $CR(f)$ should contain a distal pair for $f$. Since $f|_{X_\infty}\colon X_\infty\to X_\infty$ is proximal, it follows that $f$ is proximal and so regionally proximal.

On the other hand, we see that $f$ is {\em positively expansive}, i.e., there exists $e>0$ such that any $x,y\in X$ with $x\ne y$ satisfies $d(f^i(x),f^i(y))>e$ for some $i\ge0$, so because $X$ is perfect, $f$ is sensitive. Since $f$ is sensitive and regionally proximal, by Lemma 3.4, $f$ exhibits $gu{\rm C}$.

Assume that $(x,y)\in{\rm DC2}_\delta(X,f)$ for some $\delta>0$. Then, since
\[
\lim_{i\to\infty}d(f^i(x),X_\infty)=\lim_{i\to\infty}d(f^i(y),X_\infty)=0,
\]
by taking $x_i,y_i\in X_\infty$, $i\ge0$, with $d(f^i(x),x_i)=d(f^i(x),X_\infty)$ and $d(f^i(y),y_i)=d(f^i(y),X_\infty)$, we obtain limit-pseudo orbits $\xi=(x_i)_{i\ge0},\xi'=(y_i)_{i\ge0}$ of $f|_{X_\infty}$. Note that because $\sigma|_{X_\infty}=f|_{X_\infty}$, $\xi$ and $\xi'$ are also limit-pseudo orbits of $\sigma|_{X_\infty}$. Since $\sigma$ has the limit-shadowing property, they are limit-shadowed by $z$ and $w$, respectively. By the choice of $\xi$ and $\xi'$, we obtain
\[
(z,w)\in W^s(X_\infty,\sigma)\times W^s(X_\infty,\sigma)
\]
and $(z,w)\in{\rm DC2}(X,\sigma)$, which contradict Lemma 4.1, thus ${\rm DC2}(X,f)=\emptyset$.
\end{ex}
 
The next example is a modification of Example 4.1. In particular, it shows that $gu{\rm C}$ and the property (2) in Theorem 1.2 do not necessarily imply $gu{\rm DC1}$ in absence of the shadowing property.

\begin{ex}
\normalfont
Let $X=\{0,1\}^\mathbb{N}$ and $f\colon X\to X$ be the map defined in Example 4.1. Put
\[
\phi(x)=1+\sum_{i\ge1}2^{-i}x_i
\]
for all $x=(x_i)_{i\ge1}\in X$. Note that $\phi(0^\infty)=1$ and $\phi(X)\subset[1,2]$. We define $F\colon X\times[0,1]\to X\times[0,1]$ by
\[
F(x,t)=(f(x),t^{\phi(x)})
\]
for all $(x,t)\in X\times[0,1]$. Note that
\[
F^i(x,t)=(f^i(x),t^{\phi(x)\phi(f(x))\cdots\phi(f^{i-1}(x))})
\]
for all $(x,t)\in X\times[0,1]$ and $i\ge1$.

Since $f$ is sensitive, $F$ is also sensitive. Putting
\[
Y=\{x=(x_i)_{i\ge1}\in X\colon\text{$\{i\ge1\colon x_i=1\}$ is an infinite set}\},
\]
we see that
\[
\lim_{i\to\infty}d(F^i(x,t),X\times\{0\})=0
\]
for all $(x,t)\in Y\times[0,1)$. Since $f$ is proximal, it follows that $F|_{Y\times[0,1)}$ is proximal. Since
\[
X\times[0,1]=\overline{Y\times[0,1)},
\]
we conclude that $F$ is regionally proximal. Thus, by Lemma 3.4, $F$ exhibits $gu{\rm C}$.

Note that  
\[
\lim_{i\to\infty}d(F^i(x,t),X\times\{a(x,t)\})=0
\]
holds for all $(x,t)\in X\times[0,1]$ for some $a(x,t)\in[0,1]$. For any $p=(x,t),q=(y,s)\in X\times[0,1]$, if $a(p)\ne a(q)$, then $(p,q)$ is a distal pair for $F$. If $a(p)=a(q)$, then
\[
(p,q)\not\in{\rm DC1}(X\times[0,1],F)
\]
because ${\rm DC1}(X,f)=\emptyset$ and so $(x,y)\not\in{\rm DC1}(X,f)$. Thus,
\[
{\rm DC1}(X\times[0,1],F)=\emptyset.
\]

We have $F(0^\infty,t)=(0^\infty,t)$, i.e., $(0^\infty,t)$ is a fixed point for $F$ for all $t\in[0,1]$. This particularly implies that any pair of distinct points in $\{0^\infty\}\times[0,1]$ is a distal pair for $F$. We can show that $CR(F)$ consists of a single chain component
\[
CR(f)\times[0,1]=X_\infty\times[0,1]
\]
for $F$. This implies
\[
\mathcal{C}^s(F)=\{X_\infty\times[0,1]\}.
\]
Note that $X_\infty\times[0,1]$ contains $\{0^\infty\}\times[0,1]$ and so many distal pairs for $F$. 
\end{ex}

Then, the next example gives a continuous map $f\colon X\to X$ with the following properties:
\begin{itemize}
\item $f$ has the shadowing property,
\item $f$ exhibits $gu{\rm C}$,
\item $f$ has the limit-shadowing property, but its restriction to $CR(f)$ does not have the limit-shadowing property,   
\item $f$ exhibits $du{\rm DC2}$ and $d{\rm DC1}$ but does not exhibit $g{\rm DC2}$.
\end{itemize}
Note that the third property gives an answer to a question in \cite{K1}.

\begin{ex}
\normalfont
Let $\sigma\colon[-1,1]^\mathbb{N}\to[-1,1]^\mathbb{N}$ be the shift map, and let $d$ be the metric on $[-1,1]^\mathbb{N}$ defined by
\[
d(x,y)=\sup_{i\ge1}2^{-i}|x_i-y_i|
\]
for all $x=(x_i)_{i\ge1},y=(y_i)_{i\ge1}\in[-1,1]^\mathbb{N}$. Let $s=(s_k)_{k\ge 1}$ be a sequence of numbers with $0<s_1<s_2<\cdots$ and $\lim_{k\to\infty}s_k=1$. Put
\[
S=\{-1,1\}\cup\{-s_k\colon k\ge1\}\cup\{s_k\colon k\ge1\},
\]
a closed subset of $[-1,1]$.

We define a closed $\sigma$-invariant subset $X$ of $S^\mathbb{N}$ by for any $x=(x_i)_{i\ge1}\in S^\mathbb{N}$, $x\in X$ if and only if the following properties hold:
\begin{itemize}
\item[(1)] $|x_1|\le|x_2|\le\cdots$,
\item[(2)] For any $i\ge1$ and $k\ge1$, if $x_i=s_k$, then $x_{i+j}=-s_k$ for every $1\le j\le k$,
\item[(3)] For every $i\ge1$, if $x_i=1$, then $x_{i+j}=-1$ for all $j\ge1$.
\end{itemize}

Let $f=\sigma|_X\colon X\to X$, let $X_k=X\cap\{-s_k,s_k\}^\mathbb{N}$ for each $k\ge1$, and let
\[
X_\infty=X\cap\{-1,1\}^\mathbb{N}=\{(-1)^\infty,1(-1)^\infty\}\cup\{(-1)^m1(-1)^\infty\colon m\ge1\}.
\]
Note that $X_k$, $k\ge1$, are mixing subshifts of finite type. Put
\[
S_k=\{x=(x_i)_{i\ge1}\in X\colon |x_1|\ge s_k\},
\]
a clopen $f$-invariant subset of $X$, for each $k\ge1$, and note that $S_1\supset S_2\supset\cdots$. For any $x=(x_i)_{i\ge1}\in X$, if $|x_a|<|x_{a+1}|$ for some $a\ge1$, then taking $k\ge1$ with $|x_a|<s_k\le|x_{a+1}|$, we obtain $x\notin S_k$ and $f^a(x)\in S_k$. This implies $x\not\in CR(f)$ and so
\[
CR(f)=\{x=(x_i)_{i\ge1}\in X\colon|x_1|=|x_2|=\cdots\}=X_\infty\cup\bigcup_{k\ge1}X_k.
\]
We see that $\mathcal{C}(f)=\{X_\infty\}\cup\{X_k\colon k\ge1\}$ and $\mathcal{C}^s(f)=\{X_\infty\}$.

For each $k\ge1$, let
\[
Y_k=X\cap\{-s_1,-s_2,\dots,-s_k,s_1,s_2,\dots,s_k\}^\mathbb{N},
\]
a closed $f$-invariant subset of $X$. It is obvious that $Y_1\subset Y_2\subset\cdots$. Fix a sequence $(\delta_k)_{k\ge1}$ of numbers with $0<\delta_1>\delta_2>\cdots$ and $\lim_{k\to\infty}\delta_k=0$. Given any $x=(x_i)_{i\ge1}\in X$ and $k\ge1$, let
\begin{equation*}
y_i=
\begin{cases}
x_i&\text{if $|x_i|\le s_k$}\\
\frac{x_i}{|x_i|}\cdot s_k&\text{if $|x_i|>s_k$}
\end{cases}
\end{equation*}
for all $i\ge1$. Note that
\[
y=(y_i)_{i\ge1}\in\{-s_1,-s_2,\dots,-s_k,s_1,s_2,\dots,s_k\}^\mathbb{N}.
\]
For every $i\ge1$, if $|x_i|\le|x_{i+1}|\le s_k$, then $|y_i|=|x_i|\le|x_{i+1}|=|y_{i+1}|$. If $|x_i|\le s_k<|x_{i+1}|$, then $|y_i|=|x_i|\le s_k=|y_{i+1}|$. If $s_k<|x_i|\le|x_{i+1}|$, then $|y_i|=|y_{i+1}|=s_k$. It follows that $|y_1|\le|y_2|\le\cdots$. For every $i\ge1$, if $y_i=s_l$ and $1\le l< k$, then $x_i=s_l$ and so
\[
x_{i+1}x_{i+2}\cdots x_{i+l}=(-s_l)^l=y_{i+1}y_{i+2}\cdots y_{i+l}.
\]
If $y_i=s_k$ and $|x_i|\le s_k$, then $x_i=s_k$ and so
\[
x_{i+1}x_{i+2}\cdots x_{i+k}=(-s_k)^k=y_{i+1}y_{i+2}\cdots y_{i+k}.
\]
If $y_i=s_k$ and $|x_i|>s_k$, then $x_i\in\{1\}\cup\{s_{k+1},s_{k+2}\cdots\}$ and so 
\[
x_{i+1}x_{i+2}\cdots x_{i+k}=(-x_i)^k
\]
implying $y_{i+1}y_{i+2}\cdots y_{i+k}=(-s_k)^k$. It follows that $y\in Y_k$. By the definition of $y$, we easily see that if $s_k$ is sufficiently close to $1$, then $d(x,y)\le\delta_k$ for all $x\in X$, i.e., $X$ is contained in the $\delta_k$-neighborhood of $Y_k$.

We shall show that a choice of $s=(s_k)_{k\ge1}$ ensures the shadowing property of $f$. Fix a sequence $(\epsilon_k)_{k\ge1}$ of numbers with $0<\epsilon_1>\epsilon_2>\cdots$ and $\lim_{k\to\infty}\epsilon_k=0$. Fix any $0<s_0<1$ and assume that $s_{k-1}$, $k\ge1$, is given. Since $Y_k$ is a subshift of finite type of order $k+1$, 
\[
f|_{Y_k}\colon Y_k\to Y_k
\]
has the shadowing property, so there is $\delta'_k>0$ (independent of $s_k\in(s_{k-1},1)$) for which every $\delta'_k$-pseudo orbit $(x_i)_{i\ge0}$ of $f|_{Y_k}$ is $\epsilon_k/2$-shadowed by some $x\in Y_k$. For any $0<\delta_k<\epsilon_k/2$, if $s_k\in(s_{k-1},1)$ is sufficiently close to $1$, then as shown above, $X$ is contained in the $\delta_k$-neighborhood of $Y_k$. Then, for every $\delta_k$-pseudo orbit $(y_i)_{i\ge0}$ of $f$, we have
\[
d(x_i,y_i)=d(y_i,Y_k)\le\delta_k
\]
for all $i\ge0$ for some $x_i\in Y_k$. Since
\[
d(f(x_i),x_{i+1})\le d(f(x_i),f(y_i))+d(f(y_i),y_{i+1})+d(y_{i+1},x_{i+1})
\]
for every $i\ge0$, if $\delta_k$ is small enough, then $(x_i)_{i\ge0}$ is a $\delta'_k$-pseudo orbit of $Y_k$, $\epsilon_k/2$-shadowed by some $x\in Y_k$. This implies
\[
d(f^i(x),y_i)\le d(f^i(x),x_i)+d(x_i,y_i)\le\epsilon_k/2+\epsilon_k/2=\epsilon_k
\]
for all $i\ge0$, i.e., $(y_i)_{i\ge0}$ is $\epsilon_k$-shadowed by $x$. By defining $s=(s_k)_{k\ge1}$ in this way, we see that $f$ has the shadowing property.
Since $\mathcal{C}^s(f)=\{X_\infty\}$, $f|_{X_\infty}$ is chain mixing, and $X_\infty$ is not a singleton, by Lemma 3.2 and Theorem 1.1, $f$ exhibits $gu{\rm C}$.

Next, we shall show that $f$ has the limit-shadowing property. Let $\xi=(x_i)_{i\ge0}$ be a limit-pseudo orbit of $f$. Then, we have
\[
\lim_{i\to\infty}d(x_i,C)=0
\]
for some $C\in\mathcal{C}(f)$, and so by taking $y_i\in C$, $i\ge0$, with $d(x_i,y_i)=d(x_i,C)$, we obtain a limit-pseudo orbit $\xi'=(y_i)_{i\ge0}$ of $f|_C$ with
\[
\lim_{i\to\infty}d(x_i,y_i)=0.
\]
If $C=X_k$ for some $k\ge1$, then since $X_k$ is a subshift of finite type, and so $f|_{X_k}$ has the limit-shadowing property, there is $x\in X_k$ with $\lim_{i\to\infty}d(f^i(x),y_i)=0$. This implies $\lim_{i\to\infty}d(f^i(x),x_i)=0$, i.e., $\xi$ is limit-shadowed by $x$.  If $C=X_\infty$, then since the shift map
\[
g\colon\{-1,1\}^\mathbb{N}\to\{-1,1\}^\mathbb{N}
\]
has the limit-shadowing property, there is $z=(z_i)_{i\ge1}\in\{-1,1\}^\mathbb{N}$ such that
\[
\lim_{i\to\infty}d(g^i(z),y_i)=0.
\]
Since $\lim_{i\to\infty}d(g^i(z),X_\infty)=0$, putting
\[
\{i\ge1\colon z_i=1\}=\{i_j\colon j\ge1\},
\]
where $i_1<i_2<\cdots$, we have $i_{j+1}-i_j\to\infty$ as $j\to\infty$. Then, we see that there is $x\in X$ such that
\[
\lim_{i\to\infty}d(f^i(x),g^i(z))=0.
\]
This implies $\lim_{i\to\infty}d(f^i(x),y_i)=0$ and so $\lim_{i\to\infty}d(f^i(x),x_i)=0$, i.e., $\xi$ is limit-shadowed by $x$. Since $\xi$ is arbitrary, we conclude that $f$ satisfies the limit-shadowing property.

On the other hand, if $f|_{CR(f)}\colon CR(f)\to CR(f)$ has the limit-shadowing property, then for every $C\in\mathcal{C}(f)$, $f|_C$ should have the limit-shadowing property. However, since $f|_{X_\infty}$ does not have the limit-shadowing property, this is not the case. In other words, we find that $f|_{CR(f)}$ does not have the limit-shadowing property. 

Let us show that $f$ exhibits $du{\rm DC2}$ and $d{\rm DC1}$ but does not exhibit $g{\rm DC2}$. Take $0<e<s_1$ and note that $d(x,y)\ge e$ holds for any $k\ge1$ and $x=(x_i)_{i\ge1}$, $y=(y_i)_{i\ge1}\in X_k$ with $x_1\ne y_1$. Then, we see that
\[
{\rm DC2}_e(X_k,f|_{X_k})\ne\emptyset
\]
for every $k\ge1$. Also, we see that for any $1\le k<l$, $x_k\in Y_k$, $y_l\in X_l$, and $\epsilon>0$, there is $x\in X$ such that $d(x_k,x)\le\epsilon$ and $\lim_{i\to\infty}d(f^i(x),f^i(y_l))=0$. These properties imply that for any $0<\delta<e$ and $k\ge1$,
\[
Y_k\times Y_k\subset\overline{{\rm DC2}_\delta(X,f)}.
\]
On the other hand, as shown above, we have $Y_1\subset Y_2\subset\cdots$ and
\[
X=\overline{\bigcup_{k\ge1}Y_k}.
\]
It follows that for any $0<\delta<e$,
\[
X\times X\subset\overline{{\rm DC2}_\delta(X,f)},
\]
implying that $f$ exhibits $du{\rm DC2}$.

For any $k\ge1$, note that $\mathcal{C}^s(f|_{Y_k})=\{X_k\}$, $f|_{X_k}$ is mixing, and $X_k$ contains a distal pair for $f|_{X_k}$. Since $f|_{Y_k}$ has the shadowing property, by Lemma 3.2 and Theorem 1.2, $f|_{Y_k}$ exhibits $gu{\rm DC1}$, so especially, we have
\[
Y_k\times Y_k\subset\overline{{\rm DC1}(Y_k,f|_{Y_k})}\subset\overline{{\rm DC1}(X,f)}.
\]
It follows that
\[
X\times X\subset\overline{{\rm DC1}(X,f)},
\]
i.e., $f$ exhibits $d{\rm DC1}$.

If $f$ exhibits $g{\rm DC2}$, ${\rm DC2}(X,f)$ is a residual subset of $X^2$. Since $f$ has the shadowing property and satisfies $\mathcal{C}^s(f)=\{X_\infty\}$, as in the proof of Lemma 3.5, $W^s(X_\infty,f)\times W^s(X_\infty,f)$ is a dense $G_\delta$-subset of $X^2$, so
\[
{\rm DC2}(X,f)\cap[W^s(X_\infty,f)\times W^s(X_\infty,f)]
\]
is a residual subset of $X^2$. Take any
\[
(x,y)\in{\rm DC2}(X,f)\cap[W^s(X_\infty,f)\times W^s(X_\infty,f)].
\]
Then, taking $x_i,y_i\in X_\infty$, $i\ge0$, with $d(f^i(x),x_i)=d(f^i(x),X_\infty)$ and $d(f^i(y),y_i)=d(f^i(y),X_\infty)$, we obtain limit-pseudo orbits $\xi=(x_i)_{i\ge0}$ and $\xi'=(y_i)_{i\ge0}$ of the shift map
\[
g\colon\{-1,1\}^\mathbb{N}\to\{-1,1\}^\mathbb{N}.
\]
Since $g$ has the limit-shadowing property, $\xi$ and $\xi'$ are limit-shadowed by $z$ and $w$, respectively. By the choice of $\xi$ and $\xi'$, we obtain
\[
(z,w)\in{\rm DC2}(\{-1,1\}^\mathbb{N},g)\cap[W^s(X_\infty,g)\times W^s(X_\infty,g)],
\]
which contradicts Lemma 4.1. Thus, $f$ does not exhibit $g{\rm DC2}$.   
\end{ex}

The next example gives a continuous map with the shadowing property exhibiting $gu{\rm DC1}$. Note that its unique chain stable chain component is very simple.

\begin{ex}
\normalfont
Let $\sigma\colon[-1,1]^\mathbb{N}\to[-1,1]^\mathbb{N}$, $d$, $s=(s_k)_{k\ge1}$, and $S$ as before.  We define a closed $\sigma$-invariant subset $X$ of $S^\mathbb{N}$ by for any $x=(x_i)_{i\ge1}\in S^\mathbb{N}$, $x\in X$ if and only if the following properties hold:
\begin{itemize}
\item[(1)] $|x_1|\le|x_2|\le\cdots$,
\item[(2)] For any $i\ge1$ and $k\ge1$, if $x_i<0<x_{i+1}=s_k$, then $x_{i+j}=s_k$ for every $1\le j\le k$,
\item[(3)] For every $i\ge1$, if $x_i<0<x_{i+1}=1$, then $x_{i+j}=1$ for all $j\ge1$,
\item[(4)] For any $i\ge1$ and $k\ge1$, if $-s_k=x_{i+1}<0<x_i$, then $x_{i+j}=-s_k$ for every $1\le j\le k$,
\item[(5)] For every $i\ge1$, if $-1=x_{i+1}<0<x_i$, then $x_{i+j}=-1$ for all $j\ge1$.
\end{itemize}

Let $f=\sigma|_X\colon X\to X$, let $X_k=X\cap\{-s_k,s_k\}^\mathbb{N}$ for each $k\ge1$, and let
\[
X_\infty=X\cap\{-1,1\}^\mathbb{N}=\{(-1)^\infty,1^\infty\}\cup\{(-1)^m1^\infty\colon m\ge1\}\cup\{1^m(-1)^\infty\colon m\ge1\}.
\]
Note that $X_k$, $k\ge1$, are mixing subshifts of finite type. Similarly as in Example 4.3, we see that
\[
CR(f)=\{x=(x_i)_{i\ge1}\in X\colon|x_1|=|x_2|=\cdots\}=X_\infty\cup\bigcup_{k\ge1}X_k,
\]
$\mathcal{C}(f)=\{X_\infty\}\cup\{X_k\colon k\ge1\}$, and $\mathcal{C}^s(f)=\{X_\infty\}$.

Let
\[
Y_k=X\cap\{-s_1,-s_2,\dots,-s_k,s_1,s_2,\dots,s_k\}^\mathbb{N},
\]
$k\ge1$, a closed $f$-invariant subset of $X$, and note that $Y_1\subset Y_2\subset\cdots$. Fix a sequence $(\delta_k)_{k\ge1}$ of numbers with $0<\delta_1>\delta_2>\cdots$ and $\lim_{k\to\infty}\delta_k=0$. Given any $x=(x_i)_{i\ge1}\in X$ and $k\ge1$, let
\begin{equation*}
y_i=
\begin{cases}
x_i&\text{if $|x_i|\le s_k$}\\
\frac{x_i}{|x_i|}\cdot s_k&\text{if $|x_i|>s_k$}
\end{cases}
\end{equation*}
for all $i\ge1$. Note that
\[
y=(y_i)_{i\ge1}\in\{-s_1,-s_2,\dots,-s_k,s_1,s_2,\dots,s_k\}^\mathbb{N}.
\]
Similarly as in Example 4.3, we see that $|y_1|\le|y_2|\le\cdots$. For every $i\ge1$, if $y_i<0<y_{i+1}=s_l$ and $1\le l<k$, then $x_i<0<x_{i+1}=s_l$ and so
\[
x_{i+1}x_{i+2}\cdots x_{i+l}=(s_l)^l=y_{i+1}y_{i+2}\cdots y_{i+l}.
\]
If $y_i<0<y_{i+1}=s_k$ and $|x_{i+1}|\le s_k$, then $x_i<0<x_{i+1}=s_k$ and so
\[
x_{i+1}x_{i+2}\cdots x_{i+k}=(s_k)^k=y_{i+1}y_{i+2}\cdots y_{i+k}.
\]
If $y_i<0<y_{i+1}=s_k$ and $|x_{i+1}|>s_k$, then $x_i<0<x_{i+1}$ and $x_{i+1}\in\{1\}\cup\{s_{k+1},s_{k+2},\dots\}$, so
\[
x_{i+1}x_{i+2}\cdots x_{i+k}=(x_{i+1})^k,
\]
implying $y_{i+1}y_{i+2}\cdots y_{i+k}=(s_k)^k$. If $-s_l=y_{i+1}<0<y_i$ and $1\le l<k$, then $-s_l=x_{i+1}<0<x_i$ and so
\[
x_{i+1}x_{i+2}\cdots x_{i+l}=(-s_l)^l=y_{i+1}y_{i+2}\cdots y_{i+l}.
\]
If $-s_k=y_{i+1}<0<y_i$ and $|x_{i+1}|\le s_k$, then $-s_k=x_{i+1}<0<x_i$ and so
\[
x_{i+1}x_{i+2}\cdots x_{i+k}=(-s_k)^k=y_{i+1}y_{i+2}\cdots y_{i+k}.
\]
If $-s_k=y_{i+1}<0<y_i$ and $|x_{i+1}|>s_k$, then $x_{i+1}<0<x_i$ and $x_{i+1}\in\{-1\}\cup\{-s_{k+1},-s_{k+2},\dots\}$, so
\[
x_{i+1}x_{i+2}\cdots x_{i+k}=(x_{i+1})^k,
\]
implying $y_{i+1}y_{i+2}\cdots y_{i+k}=(-s_k)^k$. It follows that $y\in Y_k$. By the definition of $y$, we easily see that if $s_k$ is sufficiently close to $1$, then $d(x,y)\le\delta_k$ for all $x\in X$, i.e., $X$ is contained in the $\delta_k$-neighborhood of $Y_k$.

Since $Y_k$, $k\ge1$, is a subshift of finite type of order $k+1$, $f|_{Y_k}$ has the shadowing property. By a similar argument as in Example 4.3, we can show that a choice of $s=(s_k)_{k\ge1}$ ensures the shadowing property of $f$. Since $\mathcal{C}^s(f)=\{X_\infty\}$, $f|_{X_\infty}$ is chain mixing, and $X_\infty$ contains a distal pair $((-1)^\infty,1^\infty)$ for $f|_{X_\infty}$, by Lemma 3.2 and Theorem 1.2, $f$ exhibits $gu{\rm DC1}$.
\end{ex}

The final example shows that $d{\rm DC1}$ does not always imply $g{\rm C}$ even for continuous maps with the shadowing property.

\begin{ex}
\normalfont
Let $\sigma\colon[-1,1]^\mathbb{N}\to[-1,1]^\mathbb{N}$ and $d$ as before. Let $s=(s_k)_{k\ge 1}$ be a sequence of numbers with $1>s_1>s_2>\cdots$ and $\lim_{k\to\infty}s_k=0$. Put
\[
S=\{0\}\cup\{-s_k\colon k\ge1\}\cup\{s_k\colon k\ge1\},
\]
a closed subset of $[-1,1]$.

We define a closed $\sigma$-invariant subset $X$ of $S^\mathbb{N}$ by
\[
X=\{x=(x_i)_{i\ge1}\in S^\mathbb{N}\colon|x_1|\ge|x_2|\ge\cdots\}.
\]
Let $f=\sigma|_X\colon X\to X$, let $X_k=\{-s_k,s_k\}^\mathbb{N}$ for each $k\ge1$, and let $X_0=\{0^\infty\}$.
Note that $X_k$, $k\ge1$, are mixing subshifts of finite type. Put
\[
S_k=\{x=(x_i)_{i\ge1}\in X\colon |x_1|\le s_k\},
\]
$k\ge1$, a clopen $f$-invariant subset of $X$, and note that $S_1\supset S_2\supset\cdots$. For any $x=(x_i)_{i\ge1}\in X$, if $|x_a|>|x_{a+1}|$ for some $a\ge1$, then taking $k\ge1$ with $|x_a|>s_k\ge|x_{a+1}|$, we obtain $x\notin S_k$ and $f^a(x)\in S_k$. This implies $x\not\in CR(f)$ and so
\[
CR(f)=\{x=(x_i)_{i\ge1}\in X\colon|x_1|=|x_2|=\cdots\}=X_0\cup\bigcup_{k\ge1}X_k.
\]
We see that $\mathcal{C}(f)=\{X_0\}\cup\{X_k\colon k\ge1\}$ and $\mathcal{C}^s(f)=\{X_0\}$.

For each $k\ge1$, let
\[
Y_k=X\cap(\{0\}\cup\{-s_1,-s_2,\dots,-s_k,s_1,s_2,\dots,s_k\})^\mathbb{N},
\]
a closed $f$-invariant subset of $X$. It is obvious that $Y_1\subset Y_2\subset\cdots$. Fix a sequence $(\delta_k)_{k\ge1}$ of numbers with $0<\delta_1>\delta_2>\cdots$ and $\lim_{k\to\infty}\delta_k=0$. Given any $x=(x_i)_{i\ge1}\in X$ and $k\ge1$, let
\begin{equation*}
y_i=
\begin{cases}
x_i&\text{if $|x_i|\ge s_k$}\\
0&\text{if $|x_i|<s_k$}\\
\end{cases}
\end{equation*}
for all $i\ge1$. Note that
\[
y=(y_i)_{i\ge1}\in(\{0\}\cup\{-s_1,-s_2,\dots,-s_k,s_1,s_2,\dots,s_k\})^\mathbb{N}.
\]
For every $i\ge1$, if $|x_i|\ge|x_{i+1}|\ge s_k$, then $|y_i|=|x_i|\ge|x_{i+1}|=|y_{i+1}|$. If $|x_i|\ge s_k>|x_{i+1}|$, then $|y_i|=|x_i|\ge s_k>0=|y_{i+1}|$. If $s_k>|x_i|$, then $|y_i|=|y_{i+1}|=0$. It follows that $|y_1|\ge|y_2|\ge\cdots$ and so $y\in Y_k$. By the definition of $y$, we easily see that if $s_{k+1}$ is sufficiently close to $0$, then $d(x,y)\le\delta_k$ for all $x\in X$, i.e., $X$ is contained in the $\delta_k$-neighborhood of $Y_k$.

We shall show that a choice of $s=(s_k)_{k\ge1}$ ensures the shadowing property of $f$. Fix a sequence $(\epsilon_k)_{k\ge1}$ of numbers with $0<\epsilon_1>\epsilon_2>\cdots$ and $\lim_{k\to\infty}\epsilon_k=0$. Fix any $0<s_1<1$ and assume that $s_k$, $k\ge1$, is given. Since $Y_k$ is a subshift of finite type of order $2$, 
\[
f|_{Y_k}\colon Y_k\to Y_k
\]
has the shadowing property, so there is $\delta'_k>0$ for which every $\delta'_k$-pseudo orbit $(x_i)_{i\ge0}$ of $f|_{Y_k}$ is $\epsilon_k/2$-shadowed by some $x\in Y_k$. For any $0<\delta_k<\epsilon_k/2$, if $s_{k+1}\in(0,s_k)$ is sufficiently close to $0$, then as shown above, $X$ is contained in the $\delta_k$-neighborhood of $Y_k$. For every $\delta_k$-pseudo orbit $(y_i)_{i\ge0}$ of $f$, we have $d(x_i,y_i)=d(y_i,Y_k)\le\delta_k$ for all $i\ge0$ for some $x_i\in Y_k$. Then, a similar argument as in Example 4.3 shows that $(y_i)_{i\ge0}$ is $\epsilon_k$-shadowed by some $x\in Y_k$. By defining $(s_k)_{k\ge1}$ in this way, we see that $f$ has the shadowing property.

Putting
\[
Z_k=X\cap\{-s_1,-s_2,\dots,-s_k,s_1,s_2,\dots,s_k\}^\mathbb{N},
\]
$k\ge1$, a closed $f$-invariant subset of $X$, we see that $Z_1\subset Z_2\subset\cdots$ and
\[
X=\overline{\bigcup_{k\ge1}Z_k}.
\]
For any $k\ge1$, since $Z_k$ is a subshift of finite type of order $2$, $f|_{Z_k}$ has the shadowing property. Then, for every $k\ge1$, since $\mathcal{C}^s(f|_{Z_k})=\{X_k\}$, $f|_{X_k}$ is mixing, and $X_k$ contains a distal pair $((-s_k)^\infty,(s_k)^\infty)$ for $f|_{X_k}$, by Lemma 3.2 and Theorem 1.2, $f|_{Z_k}$ exhibits $gu{\rm DC1}$; therefore,
\[
Z_k\times Z_k\subset\overline{{\rm DC1}(Z_k,f|_{Z_k})}\subset\overline{{\rm DC1}(X,f)}.
\]
It follows that
\[
X\times X\subset\overline{{\rm DC1}(X,f)},
\]
i.e., $f$ exhibits $d{\rm DC1}$. However, since the unique $X_0\in\mathcal{C}^s(f)$ is a singleton, by Theorem 1.1, $f$ does not exhibit $g{\rm C}$.
\end{ex}

\section*{Acknowledgements}

This work was supported by JSPS KAKENHI Grant Number JP20J01143.

\end{document}